\documentclass[11pt]{amsart}
\usepackage{graphicx}
\usepackage{amssymb}
\usepackage{amsmath}
\usepackage{amsthm,amsfonts,bbm}
\usepackage{amscd}
\usepackage{geometry}
\usepackage[all,2cell]{xy}
\usepackage{epsfig,epstopdf}
\usepackage[backref]{hyperref}

\UseAllTwocells \SilentMatrices

\newtheorem{thm}{Theorem}[section]

\newtheorem{lem}[thm]{Lemma}

\theoremstyle{definition}
\newtheorem{defi}[thm]{Definition}
\theoremstyle{remark}
\newtheorem{rmk}[thm]{\bf Remark}
\newtheorem{exm}[thm]{\bf Example}
\numberwithin{equation}{section}
\numberwithin{figure}{section}
\def\C{\mathbb{C}}

\def \d{\delta}

\def \dim{\text{dim}}
\def \dw{\text{down}}

\def \im{\text{im~}}

\def \K{\mathcal{K}}
\def \L{\mathcal{L}}
\def\la{\lambda}
\def \omg{\omega}
\def \p{\partial}

\def \R{\mathbb{R}}

\def \spec{\text{~Spec}}

\def \up{\text{up}}

\DeclareMathOperator{\sgn}{sgn}

\begin{document}
\title[Spectra of Laplace operators]{The spectra of Laplace operators on covering simplicial complexes}

\author[Y.-Z. Fan]{Yi-Zheng Fan*}
\address{Center for Pure Mathematics, School of Mathematical Sciences, Anhui University, Hefei 230601, P. R. China}
\email{fanyz@ahu.edu.cn}
\thanks{*The corresponding author. Supported by National Natural Science Foundation of China (Grant No. 12331012).}

\author[Y.-M. Song]{Yi-Min Song}
\address{School of Mathematics and Physics, Anhui Jianzhu University, Hefei 230601, P. R. China}
\email{songym@ahjzu.edu.cn}

\author[Y. Wang]{Yi Wang$^\sharp$}
\address{Center for Pure Mathematics, School of Mathematical Sciences, Anhui University, Hefei 230601, P. R. China}
\email{wangy@ahu.edu.cn}
\thanks{$^\sharp$Supported by National Natural Science Foundation of China (Grant No. 12171002).}



\begin{abstract}
We give a decomposition of the Laplace operator (in matrix form) of a covering simplicial complex as a direct sum of several matrices, one of which is the Laplace operator of the base complex. It follows that the spectrum of a covering simplicial complex is a multiset union of the spectrum of the base simplicial complex and the spectra of other relevant matrices, which implies the spectral inclusion property of Horak and Jost.
In the case of a $2$-fold covering, we show that the spectrum is a multiset union of the spectrum of the base complex and that of an incidence-signed simplicial complex, thereby generalizing a result of Bilu and Linial from graphs to simplicial complexes. Additionally, we show that the dimension of the cohomology of a covering complex is greater than or equal to that of the base complex. Our arguments exploit the coverings of incidence graphs of simplicial complexes and the representation theory of permutation groups.
\end{abstract}

\maketitle

\maketitle

\section{Introduction}
The study of graph Laplacians has a long and prolific history.
It first emerged in a paper by Kirchhoff \cite{Kirch}, where he analyzed electrical networks and formulated the celebrated matrix tree theorem.
In the early 1970s, Fiedler \cite{Fied} established a relationship between the second smallest eigenvalue and the connectivity of a graph.
Since then, numerous papers have been published on the Laplacian spectra of graphs (see \cite{Merris}).
The normalized graph Laplacian was introduced by Bottema \cite{Bot}, who investigated a transition probability operator on graphs; for further details, refer to \cite{Chung96}.

The graph Laplacian was generalized to simplicial complexes by Eckmann \cite{Eckmann}, who established the discrete version of the Hodge theorem.
This theorem can be formulated as
$$ \text{ker}(\d_i^* \d_i+ \d_{i-1} \d_{i-1}^*) = \tilde{H}^i(K,\mathbb{R}),$$
where $L_i=  \d_i^* \d_i+ \d_{i-1} \d_{i-1}^*$ is the higher order combinatorial Laplacian.
 Subsequent efforts toward the normalization of the combinatorial Laplace operator have been made by Chung \cite{Chung93}, Taszus \cite{Tas}, Lu and
Peng \cite{LP} and Garland \cite{Garl}.
Horak and Jost further advanced this field by developing a general framework for Laplace operators defined in terms of the combinatorial structure of a  simplicial complex, including the combinatorial Laplacian and the normalized Laplacian.

The graph covering (also known as lift) has been  extensively  introduced and studied in numerous  literatures \cite{Mass,Rot,Big} from the viewpoint of topology or simplicial complex.
It was later generalized to hypergraph covering in various forms \cite{Dor,LH1,Song}.
Rotman \cite{Rot} introduced the notion of simplicial complex covering, which Gustavson \cite{Gust} employed to investigate the Laplacian spectrum.
As noted by Horak and Jost \cite{HJ13A}, there exist counterexamples to the Universal Lifting Theorem for discrete covering maps \cite[Theorem 2.1]{Rot} and the spectral inclusion theorem \cite[Theorem 4.4]{Gust}.
To address these issues, Horak and Jost \cite{HJ13A} introduced the concept of strong covering; see Section \ref{Secov} for details.

By the representation theory of permutation groups, the characteristic polynomial of the adjacency matrix of a covering graph was derived as a product of the characteristic polynomials of certain matrices \cite{Mizu, Feng}, which implies that the spectrum of the adjacency matrix (simply called the adjacency spectrum) of a covering graph contains that of its underlying graph.
Li and Hou \cite{LH2} extened this approach to formulate the characteristic polynomials of the Laplacian matrix or normalized Laplacian matrix of a covering graph.
Specifically, for $2$-fold covering (or $2$-lift) of a graph $G$, Bilu and Linial \cite{Bilu} established the following spectral union property using eigenvector approach.
It is worth noting that this property can also be demonstrated via the sign representation of the symmetric group $\mathbb{S}_2$.

\begin{thm}\label{G-2lift} \cite{Bilu}
Let $\bar{G}$ be a $2$-lift of a graph $G$.
Then the adjacency spectrum of $\bar{G}$ is a multiset union of adjacency spectrum of $G$ and the adjacency spectrum of a signed graph $G_s$ with $G$ as the underlying graph.
\end{thm}

The signed graph $G_s$ in Theorem \ref{G-2lift} is determined as follows.
According to a result of Gross and Tucker \cite{GT} (Lemma \ref{per}),
the $2$-lift $\bar{G}$ is isomorphic to a derived graph $G^\psi$ obtained from a permutation assignment $\psi: E(G) \to \mathbb{S}_2$, which assigns a permutation in $\mathbb{S}_2$ to each edge of $G$.
The signing $s$ of $G_s$ is defined by $s(e)=\sgn \psi(e)$ for each edge $e$ of $G$.

Let $K$ be a simplicial complex, and let $L_i^{\up}(K)$ and $\Delta_i^{\up}(K)$ denote the $i$-up combinatorial Laplace operator and normalized Laplace operator of $K$, respectively.
Horak and Jost \cite{HJ13A} established the following spectral inclusion property between a covering complex $K$ and its base complex $M$, where $\spec A$ denotes the spectrum of an operator or matrix $A$.

\begin{thm}\cite{HJ13B}\label{inclusion}
Let $\phi: K \to M$ be a $k$-fold covering map.
Then
$$ \spec L_i^{\up}(M) \subset \spec L_i^{\up}(K), ~ \spec \Delta_i^{\up}(M) \subset \spec \Delta_i^{\up}(K).$$
\end{thm}

We focus on the following problem:
\emph{What are the remaining eigenvalues of the covering complex except the eigenvalues of the base complex?}

To outline the method for addressing this problem, we first introduce some fundamental concepts (for a comprehensive treatment, see Section \ref{sec2}).
Let $K$ be a simplicial complex, and let $S_i(K)$ denote the set of all $i$-faces of $K$.
Let $C_i(K,\R)$ represent the $i$-th chain group of $K$, and let $\p_i: C^{i}(K,\R) \to C^{i-1}(K,\R)$ denote the boundary map of $K$.
The coboundary map $\d_i: C^{i}(K,\R) \to C^{i+1}(K,\R)$ is the conjugate of $\p_{i+1}$.
By introducing inner products on $C^i(K,\R)$ and $C^{i+1}(K,\R)$ respectively, we obtain the adjoint $\d^*_i: C^{i+1}(K,\R) \to C^i(K,\R)$ of $\d_i$.
Here, the inner product on $C^i(K,\R)$ (for each $i$) is defined such that elementary cochains are orthogonal to each other, which is equivalent to introducing a weight function $w$ on all faces of $K$.
We denote by $W_i(K)$ the diagonal matrix whose entries are the weights of $i$-faces.
The $i$-up Laplace operator of $K$ is defined as $\L_i^{up}(K)=\d_i^* \d_i$.
To avoid  ambiguity, we write $w_K$ instead of $w$,  $\d_i(K)$ instead of $\d_i$, and let $D_i(K)$  denote the matrix of $\d_i(K)$ with respect to the basis of elementary cochains.
The matrix of $\L_i^{up}(K)$ is then given by:
$$  \L_i^{up}(K)=W_i(K)^{-1} D_i(K)^\top W_{i+1}(K) D_i(K).$$

\subsection{Outline of our method}
Let $\phi: K \to M$ be a $k$-fold covering.
Our objective is to decompose the $i$-up Laplace operator $\L_i^{up}(K)$ of $K$ into a direct sum of matrices, where the $i$-up Laplace operator $\L_i^{up}(M)$ of $M$ appears as one of the summands.
Since the $i$-up Laplace operator of a simplicial complex involves only its $i$-dimensional faces and $(i+1)$-dimensional faces, it is necessary to establish a relationship between $i$-faces and $(i+1)$-faces of $K$ and those of $M$.

\subsubsection{Coverings of incidence graphs}\label{drv}
The \emph{$i$-th incidence graph} of $K$, denoted by $B_i(K)$, is a bipartite graph with vertex set $S_i(K) \cup S_{i+1}(K)$ such that $\{F, \bar{F}\}$ is an edge if and only if $F \in \p \bar{F}$, where $F \in S_i(K)$, $\bar{F} \in S_{i+1}(K)$, and $\p \bar{F}$ denotes the set of $i$-faces in the boundary of $\bar{F}$.
We show that $\phi$ induces a $k$-fold covering from $B_i(K)$ to $B_i(M)$ as follows.

\begin{lem}\label{perCom}
Let $\phi: K \to M$ be a $k$-fold covering.
Then $\phi$ induces a  $k$-fold covering from $B_i(K)$ to $B_i(M)$, and
there exists a $\psi: E(B_i(M)) \to \mathbb{S}_k$ such that $B_i(M)^\psi$ is isomorphic to $B_i(K)$ via a mapping which sends $S_i(M) \times [k]$ to $S_i(K)$ and   $S_{i+1}(M) \times [k]$ to $S_{i+1}(K)$.
\end{lem}

To clarify the derived graph $B_i(M)^\psi$ in Lemma \ref{perCom}, we introduce the following definitions.
Let $D$ be directed graph (possibly with multiple arcs), and let $\mathbb{S}_k$ denote the symmetric group on the set $[k]:=\{1,2,\ldots,k\}$.
Let $\psi:E(D)\to\mathbb{S}_k$ which assigns a permutation to each arc of $D$.
 The pair $(D, \psi)$ is referred to as a \emph{permutation voltage digraph}.
 The \emph{derived digraph} $D^\psi$ associated with $(D, \psi)$ is a directed graph with vertex set $V(D)\times[k]$ such that $((u, i),(v, j))$ is an arc of $D^\psi$ if and only if $(u, v)\in E(D)$ and $i=\psi(u, v)(j)$.
For a simple graph $G$, let $\overleftrightarrow{G}$ denote the symmetric directed graph obtained from $G$ by replacing each edge $\{u, v\}$ by two arcs with opposite directions: $e=(u, v)$ and $e^{-1}:= (v, u)$.
Let $\psi: E(\overleftrightarrow{G})\to\mathbb{S}_k$ be a permutation assignment satisfying $\psi(e)^{-1}=\psi(e^{-1})$ for each arc $e$ of $\overleftrightarrow{G}$.
By definition, the derived digraph $\overleftrightarrow{G}^\psi$ (abbreviated as $G^\psi$) has symmetric arcs and is thus regarded as a simple graph.

\subsubsection{Relations between coboundary maps}\label{Rcob}
The entries of $D_i(K)$  (the matrix of the coboundary map $\d_i(K)$) are given by
$$D_i(K)_{[\bar{F}]^*, [F]^*}=\sgn([F], \p [\bar{F}]),$$
where $F \in S_i(K)$, $\bar{F} \in S_{i+1}(K)$, and $\sgn([F], \p [\bar{F}])$ is defined in \eqref{sgn1}.
By Lemma \ref{perCom}, there exists an isomorphism
$$\eta: B_i(M)^\psi \to B_i(K)$$ that maps $S_i(M) \times [k]$ to $S_i(K)$ and   $S_{i+1}(M) \times [k]$ to $S_{i+1}(K)$.
This allows us to relabel the rows and columns of $D_i(K)$ such that
$$ D_i(K)_{\eta^{-1}(\bar{F}]), \eta^{-1}([F])}:=D_i(K)_{[\bar{F}]^*, [F]^*}=\sgn([F], \p [\bar{F}]);$$
or equivalently,
\begin{equation}\label{DiK} D_i(K)_{(\bar{G},l),(G,j)}=\sgn([\eta(G,j)], \p [\eta(\bar{G},l)]).
\end{equation}
Similarly, we relabel the rows and columns of $W_i(K)$ as follows:
\begin{equation}\label{WiK}
 W_i(K)_{(G,l),(G,l)}=w(\eta(G,l)).
\end{equation}

Let $\Lambda_{i+1}(M)^\eta$ denote the diagonal matrix with rows and columns indexed by $S_{i+1}(M) \times [k]$, defined by
\begin{equation}\label{Lmdi1p} \Lambda_{i+1}(M)^\eta_{(\bar{G},i),(\bar{G},i)}=\sgn([\eta(\bar{G},i)], [\bar{G}]),
\end{equation}
and let $\Lambda_i(M)^\eta$ denote the diagonal matrix with rows and columns indexed by $S_{i}(M)  \times [k]$, defined by
\begin{equation}\label{Lmdip} \Lambda_i(M)^\eta_{(G,i),(G,i)}=\sgn([\eta(G,i)], [G]),
\end{equation}
where the sign function is defined in \eqref{sgn2}.

Let $ D_i(M)^\psi$ be the matrix with rows  indexed by  $S_{i+1}(M) \times [k]$ and columns indexed by $S_{i}(M)  \times [k]$, defined by
\begin{equation}\label{tdilp}
 D_i(M)^\psi_{(\bar{G},l),(G,j)}=
\begin{cases}
\sgn([G], \p [\bar{G}]), & \text{if~} G \in \p \bar{G} \text{~and~} l = \psi(\bar{G},G)(j),\\
0, & \text{otherwise}.
\end{cases}
\end{equation}

\begin{lem}\label{F-DiK}
Let $\phi: K \to M$ be a $k$-fold covering map, and let $\psi: E(B_i(M) \to \mathbb{S}_k$ be the permutation assignment such that $\eta: B_i(M)^\psi \to B_i(K)$ is an isomorphism.
Then
\begin{equation}\label{matrel}
D_i(K)=\Lambda_{i+1}(M)^\eta D_i(M)^\psi \Lambda_i(M)^\eta.
\end{equation}
where $\Lambda_{i+1}(M)^\eta$, $\Lambda_i(M)^\eta$ and $D_i(M)^\psi$ are defined in (\ref{Lmdi1p}), (\ref{Lmdip})  and (\ref{tdilp}), respectively.
\end{lem}

The entries of $D_i(M)$ (the matrix of $\d_i(M)$) are given by
$$D_i(M)_{[\bar{G}]^*, [G]^*}=\sgn([G], \p_{i+1} [\bar{G}]).$$
Relabeling the rows and columns of $D_i(M)$, we have $$D_i(M)_{\bar{G},G}:=D_i(M)_{[\bar{G}]^*, [G]^*}=\sgn([G], \p_{i+1} [\bar{G}]).$$
Notably, $D_i(M)^\psi$ is closely related to $D_i(M)$,
thereby establishing an indirect relationship between $D_i(K)$ and $D_i(M)$.

\subsubsection{Representations of permutation groups}\label{Rep}
Let
\begin{equation}\label{Psi}\Psi_i=\langle \psi(G,\bar{G}): G \in S_i(M), \bar{G} \in S_{i+1}(M), G \in \p \bar{G}\rangle,
\end{equation}
 denote the subgroup of $\mathbb{S}_k$ generated by the permutations assigned by $\psi$.
For each $g \in \Psi_i$, define $D^g_i(M)=(D^g_{\bar{G}, G})$ such that
\begin{equation}\label{Dg}
D^g_{\bar{G}, G}=
\begin{cases}
\sgn([G], \p [\bar{G}]), & \text{if~} G \in \p \bar{G} \text{~and~} \psi(G, \bar{G})=g,\\
0, & \mbox{otherwise}.
\end{cases}
\end{equation}
Then
\begin{equation}\label{decompL}  D_i(M)=\sum_{g \in \Psi} D_i^g(M).
\end{equation}

The permutation representation $\varrho$ of $\Psi$ maps
each $g \in \Psi$ to a permutation matrix, namely
\begin{equation}\label{PRe}
\varrho: \Psi \to \text{GL}_k(\mathbb{C}), g \mapsto P^g=(p_{ij}^g),
\end{equation}
where
$p_{ij}^g=1$ if $i=g(j)$ and $p_{ij}^g=0$ otherwise.
This yields the decomposition
\begin{equation}\label{Dsum}
 D_i(M)^\psi=\sum_{g \in \Psi_i} D_i^g(M) \otimes P^g.
\end{equation}

By the theory of permutation group representations, let
$$ \varrho=\varrho_1 \oplus \varrho_2 \oplus \cdots \oplus \varrho_t$$
be the decomposition of $\varrho$ into a sum of irreducible sub-representations,
where $\varrho_1=1$ is the identity representation of degree one.
So, there exists an invertible matrix $T$ such that for all $g \in \Psi_i$,
\begin{equation}\label{rep}
 T^{-1} P^g T= I_1 \oplus \varrho_2(g) \oplus \cdots \oplus \rho_t(g),
\end{equation}
where $I_1$ is the identity matrix of order $1$ corresponding to $\varrho_1$, and $\varrho_2(g),\ldots, \varrho_t(g)$ are the matrices associated with   $\varrho_2,\ldots,\varrho_t$ respectively.

Here, we provide an illustration of the decomposition of $D_i(M)^\psi$ by permutation representation.
The same approach will be employed for the decomposition of $\L_i^{up}(K)$ in Section \ref{sec4}.
By \eqref{Dsum}, \eqref{rep} and \eqref{decompL}, we have
\begin{align*}
 (I \otimes T)^{-1} D_i(M)^\psi (I  \otimes T) & =
\sum_{g \in \Psi_i} D_i^g(M) \otimes (T^{-1} P^g T) \\
& = \sum_{g \in \Psi_i} D_i^g(M) \otimes \big(I_1 \oplus \varrho_2(g) \oplus \cdots \oplus \rho_t(g)\big) \\
& = \sum_{g \in \Psi_i} D_i^g(M) \oplus \bigg(\sum_{g \in \Psi_i} D_i^g(M) \otimes \varrho_2(g) \bigg) \oplus \cdots \oplus \bigg(\sum_{g \in \Psi_i} D_i^g(M) \otimes \varrho_t(g)\bigg) \\
& = D_i(M) \oplus \bigg(\sum_{g \in \Psi_i} D_i^g(M) \otimes \varrho_2(g)\bigg) \oplus \cdots \oplus \bigg(\sum_{g \in \Psi_i} D_i^g(M) \otimes \varrho_t(g)\bigg).
\end{align*}

\subsection{Main results}\label{main}
Herein, we present key results concerning the $i$-up Laplace operators. Analogous results for the $i$-down Laplace operators can be derived using a similar approach. Additionally, we characterize the special case where the permutation group $\Psi_i$ (as defined in \eqref{Psi}) is abelian.

\begin{thm}\label{mainthm}
Let $\phi: K \to M$ be a $k$-fold covering map, and  let $\psi: E(B_i(M) \to \mathbb{S}_k$ be the permutation assignment such that $\eta: B_i(M)^\psi \to B_i(K)$ is an isomorphism.
Let $\Psi_i$ be the group defined in (\ref{Psi}), whose permutation representation $\varrho$ has a decomposition as in (\ref{rep}) for some invertible matrix $T$.
Suppose that $w_K(F)=w_M(\phi(F))$ for all $F \in K$.
Then
\begin{multline*}
 (I \otimes T)^{-1} \Lambda_i(M)^\eta\L_i^{up}(K)(\Lambda_i(M)^\eta)^{-1} (I \otimes T) \\
 =\L_i^{up}(M) \oplus \bigoplus_{j=2}^t \bigg( \sum_{g \in \Psi_i} \big(W_i(M)^{-1} (D_i^g(M))^\top \big) \otimes \varrho_j(g^{-1}) \bigg) \bigg(\sum_{g \in \Psi_i} \big(W_{i+1}(M)D_i^g(M)\big) \otimes \varrho_j(g)\bigg),
\end{multline*}
where  $\Lambda_i(M)^\eta$ is  defined in (\ref{Lmdip}) and $D_i^g(M)$ is defined in (\ref{Dg}).
\end{thm}

By Theorem \ref{mainthm}, if all irreducible sub-representations $\varrho_j$ of $\Psi_i$ are known, the remaining eigenvalues of the covering complex $K$ can be determined.
In particular, this recovers the spectral inclusion property established by  Horak and Jost \cite{HJ13B} (see Theorem \ref{inclusion}).

Before introducing the spectral union property for $2$-fold covering complexes, we first introduce  incidence-signed complexes, which are analogous to signed graphs.
Let $K$ be a simplicial complex, and let $F, \bar{F} \in K$.
If $F \in \p \bar{F}$, then $(F,\bar{F})$ is called an \emph{incidence} of $K$.

\begin{defi}\label{inc-sgn}
The \emph{incidence-signed complex} is a pair $(K, s)$, where $K$ is a simplicial complex, and $s: K \times K \to \{-1,0,1\}$ satisfies
$s(F, \bar{F}) \in \{-1,1\}$ if $F \in \p \bar{F}$, and $s(F, \bar{F})=0$ otherwise.
\end{defi}

In Section \ref{sec2} we define the Laplace operators for incidence-signed complex $(K,s)$, following the same framework used for Laplace operators of ordinary complexes.
It is well-known that $\mathbb{S}_2$ has two irreducible representations: the identity representation and the sign representation.
Applying Theorem \ref{mainthm}, we obtain the following spectral union property, which extends Bilu and Linial's work \cite{Bilu} (Theorem \ref{G-2lift}) from graphs to complexes.

\begin{thm}\label{SpecUnion}
Let $\phi: K \to M$ be a $2$-fold covering map.
Suppose that $w_K(F)=w_M(\phi(F))$ for all $F \in K$.
Then, as a multiset union,
$$ \spec \L_i^{up}(K))= \spec\L_i^{up}(M) \cup \spec\L_i^{up}(M,s),$$
where $s(F, \bar{F})=\sgn \psi(F, \bar{F})$ for $(F,  \bar{F}) \in S_i(M) \times S_{i+1}(M)$ with $F \in \p \bar{F}$, and $\psi:E(B_i(M)) \to \mathbb{S}_2$ is such that $B_i(M)^\psi \cong B_i(K)$.
\end{thm}

\subsection{Organization of the paper}
In Section \ref{sec2} we introduce basic terminology, including simplicial complexes and their Laplace operators, incidence-signed complexes, incidence-weighted complexes, and the Laplace operators defined on the latter two.
In Section \ref{Secov}, we introduce the covering complexes, and prove Lemma \ref{perCom} concerning the relationship between incidence graphs of complexes.
Section \ref{sec4} is devoted to proving our main results, including the counterpart for $i$-down Laplace operators.
We also characterize the special case where the permutation group $\Psi_i$ (defined in \eqref{Psi}) is abelian: in this case, the Laplace spectrum of $K$ is a multiset union of the Laplace spectrum of $M$ and the Laplace spectra of $(k-1)$ incidence-weighted complexes.
In the final section, using the discrete Hodge theorem, we show  that the dimension of the $i$-th cohomology group of a covering complex is greater than or equal to that of its base complex.

\section{Preliminaries}\label{sec2}
\subsection{Simplicial complexes and Laplace operators}
Let $V$ be a finite set.
An \emph{abstract simplicial complex} (simply called a \emph{complex}) $K$ over $V$ is a collection of the subsets of $V$ that is closed under inclusion.
An \emph{$i$-face} or \emph{$i$-simplex} of $K$ is an element of $K$ with cardinality $i+1$.
The \emph{dimension} of an $i$-face is $i$, and the \emph{dimension} of $K$, denoted by $\dim K$, is the maximum dimension of all faces of $K$.
Faces that are maximal under inclusion are called \emph{facets}.
Thus, a complex can be viewed as a hypergraph in which the facets correspond to hyperedges.

We assume that $\emptyset \in K$,  referred to as the empty simplex with dimension $-1$.
Let $S_i(K)$ denote the set of all $i$-faces of $\K$, where $S_{-1}(K):=\{\emptyset\}$.
The $p$-skeleton of $K$, written $K^{(p)}$, is the set of all simplices of $K$ of dimension at most $p$.
In particular, $K^{(1)}\backslash \{\emptyset\}$ is an ordinary graph, where $0$-faces are  called \emph{vertices} (denoted by $V(K)$), and $1$-faces are called \emph{edges}.
A complex $K$ is said to be \emph{connected} if the graph $K^{(1)}\backslash \{\emptyset\}$ is connected.

We say a face $F$ is \emph{oriented} if we assign an ordering of its vertices and write it as $[F]$.
Two ordering of the vertices of $F$ are said to determine the \emph{same orientation} if there is an even permutation transforming one ordering into the other.
If the permutation is odd, then the orientation are opposite.
The \emph{$i$-th chain group} of $K$ over $\R$, denoted by $C_i(K,\R)$, is the vector space over $\R$ generated by all oriented $i$-faces of $K$, modulo the relation $[F_1]+[F_2]=0$, where $[F_1]$ and $[F_2]$ are opposite orientations of a same face.
The \emph{$i$-th cochain group} $C^i(K,\R)$ is the dual of $C_i(K,\R)$, i.e., $C^i(K,\R)=\text{Hom}(C_i(K,\R),\R)$, which is generated by the dual basis $\{[F]^*: F \in S_i(K)\}$, where
$ [F]^*([F])=1$, and $[F]^*([F'])=0$ for  $F' \ne F$.
The functions $[F]^*$ are called the \emph{elementary cochains}.

For each integer $i=0,1,\ldots,\dim K$, The \emph{boundary map} $\p_i: C_i(K,\R) \to C_{i-1}(K,\R)$ is defined by
$$\p_i([v_0,\ldots,v_i])=\sum_{j=0}^i (-1)^j[v_0,\ldots,\hat{v}_j,\ldots,v_i],$$
for each oriented $i$-face $[v_0,\ldots,v_i]$ of $K$,
where $\hat{v}_j$ indicates the vertex $v_j$ is omitted.
This gives rise to the \emph{augmented chain complex} of $K$:
$$  \cdots \longrightarrow C_{i+1}(K,\R) \stackrel{\p_{i+1} }{\longrightarrow} C_i(K,\R) \stackrel{\p_{i} }{\longrightarrow} C_{i-1}(K,\R) \rightarrow \cdots \longrightarrow  C_{-1}(K,\R) \longrightarrow 0,$$
satisfying $\p_i \circ \p_{i+1}=0$.

The \emph{coboundary map} $\delta_{i}: C^{i}(K,\R) \to C^{i+1}(K,\R)$ is the conjugate of $\p_{i+1}$, defined by $ \d_{i} f = f \p_{i+1}$.
Thus
$$ (\d_i f)([v_0,\ldots,v_{i+1}])=\sum_{j=0}^{i+1} (-1)^jf([v_0,\ldots,\hat{v}_j,\ldots,v_{i+1}]).$$
Correspondingly, we have the \emph{augmented cochain complex} of $K$:
$$  \cdots \longleftarrow C^{i+1}(K,\R) \stackrel{\delta_{i} }{\longleftarrow} C^i(K,\R) \stackrel{\delta_{i-1} }{\longleftarrow} C^{i-1}(K,\R) \longleftarrow \cdots  {\longleftarrow} C^{-1}(K,\R) \longleftarrow  0,$$
satisfying $\delta_{i}\circ \delta_{i-1}=0$.
The \emph{$i$-th reduced cohomology group} for every $i \ge 0$ is defined to be $$\tilde{H}^i(K,\R)=\ker \delta_i / \im \delta_{i-1}.$$

 By introducing inner products on $C^i(K,\R)$ and $C^{i+1}(K,\R)$ respectively, we obtain the \emph{adjoint} $\d^*_i: C^{i+1}(K,\R) \to C^i(K,\R)$ of $\d_i$, defined by
$$ \langle \delta_i f_1, f_2 \rangle_{C^{i+1}} =\langle f_1, \delta^*_i f_2\rangle_{C^{i}}$$ for all $f_1 \in C^i(K,\R), f_2 \in C^{i+1}(K,\R)$.

\begin{defi}\cite{HJ13B}
The following three operators are defined on $C^i(K,\R)$.

(1) The $i$-dimensional combinatorial up Laplace operator or simply the \emph{$i$-up Laplace operator}: $$  \L_i^{\up}(K):=\d_i^* \d_i.$$

(2) The $i$-dimensional combinatorial down Laplace operator or simply the \emph{$i$-down Laplace operator}:
$$ \L_i^{\dw}(K):=\d_{i-1}\d_{i-1}^*.$$

(3) The $i$-dimensional combinatorial Laplace operator or simply the \emph{$i$-Laplace operator}:
 $$\L_i(K):=\L_i^{\up}(K)+\L_i^{\dw}(K)=\delta^*_i \delta_i+\delta_{i-1}\delta^*_{i-1}.$$
 \end{defi}

Horak and Jost \cite{HJ13A, HJ13B} proposed defining an inner product on $C^i(K, \R)$ such that
the elementary cochains are orthogonal to each other.
This is equivalent to introducing a weight function on all faces of $K$:
$$ w: \bigcup_{i=-1}^{\dim K} S_i(K) \to \R^+,$$
whereby the inner product is given by
$$ \langle f,g \rangle_{C^i}=\sum_{F \in S_i(K)} w(F) f([F])g([F]).$$

In this paper, the weight $w$ is implicit in the context for the Laplace operator.
If $w \equiv 1$ on all faces, the corresponding Laplacian is the \emph{combinatorial Laplace operator}, denoted by $L_i(K)$, as discussed in \cite{Duv, Fried}.
If $w$ is identically $1$ on all facets and satisfies the normalization condition:
$ w(F)=\sum_{\bar{F} \in S_{i+1}(K): F \in \p \bar{F}}w(\bar{F})$
for every non-facet $F \in S_i(K)$,
then $w$ defines the \emph{normalized Laplace operator}, denoted by $\Delta_i(K)$, as analyzed in \cite{HJ13B}.

Horak and Jost \cite{HJ13A, HJ13B} provided explicit formulas for $\L_i^{\up}$ and $\L_i^{\dw}$.
 Here, we consider the matrix forms of these operators.
With a slight abuse of notation, for any $\bar{F} \in S_{i+1}(K)$, we let $\p \bar{F}$ denote the set of all $i$-faces in the boundary of $\bar{F}$.
Given an oriented $(i+1)$-face $[\bar{F}]:=[v_0,\ldots,v_{i+1}]$ and its $i$-dimensional boundary face $[F_j]:=[v_0,\ldots,\hat{v}_j,\ldots,v_{i+1}]$,
we define
\begin{equation}\label{sgn1}
\sgn([F_j], \p[\bar{F}]]=(-1)^j,
\end{equation}
and set $\sgn([F], \p[\bar{F}]]=0$ if $F \notin \p \bar{F}$.

Let $D_i$ be the matrix of the coboundary map $\d_i: C^i \to C^{i+1}$ with respect to the bases of $C^i$ and $C^{i+1}$ consisting of elementary cochains.
Its entries are given by
 $$(D_i)_{[\bar{F}]^*, [F]^*}=\sgn([F], \p [\bar{F}]).$$
The matrix $D_i^*$ representing the adjoint operator $ \d_i^*$ satisfies
$$ (D_i^*)_{[F]^*, [\bar{F}]^*}=\frac{w(\bar{F})}{w(F)} \sgn([F], \p [\bar{F}]).$$
Let $W_i$ denote diagonal matrix whose entries are the weights of $i$-faces, i.e.,
$(W_i)_{[F]^*, [F]^*}=w(F)$.
It follows that
$ D_i^*=W_i^{-1} D_i^\top W_{i+1},$
and consequently,
\begin{equation}\label{DefL} \L_i^{\up}(K)=W_i^{-1} D_i^\top W_{i+1}D_i,~ \L_i^{\dw}(K)=D_{i-1} W_{i-1}^{-1}D_{i-1}^\top W_i.
\end{equation}

\subsection{Incidence-signed complexes and incidence-weighted complexes}
Let $(K,s)$ be an incidence-signed complex.
The \emph{signed boundary map} $\p_i^s: C_i(K,\R) \to C_{i-1}(K,\R)$ is defined by
$$\p_i^s[v_0,\ldots,v_i]=\sum_{j=0}^i (-1)^j[v_0,\ldots,\hat{v}_j,\ldots,v_i]
s(\{v_0,\ldots,\hat{v}_j,\ldots,v_i\},\{v_0,\ldots,v_i\}).$$

The \emph{signed co-boundary map} $\d^s_{i}: C^{i}(K,\R) \to C^{i+1}(K,\R)$ is the conjugate of $\p_{i+1}^s$, i.e., for all $f \in C^{i}$,
$\d^s_{i}f = f \p_{i+1}^s$.

The \emph{signed adjoint}  $(\d_{i}^s)^*: C^{i+1}(K,\R) \to C^{i}(K,\R)$ is the adjoint of $\d_{i}^s$, satisfying
$$ ( \delta_{i}^s f_1, f_2 )_{C^{i+1}} =( f_1, (\delta^s_{i})^* f_2)_{C^{i}}$$ for all $f_1 \in C^{i}, f_2 \in C^{i+1}$.

\begin{defi}
Let $(K,s)$ be an incidence-signed complex.

(1) The $i$-up Laplace operator of $(K,s)$ is defined as  $\L_i^{\up}(K,s)=(\d_i^s)^* \d_i^s$.

(2) The $i$-down Laplace operator of $(K,s)$ is defined as $\L_i^{\dw}(K,s)=\d_{i-1}^s(\d_{i-1}^s)^*$.

(3) The $i$-Laplace operator of $(K,s)$ is defined as $\L(K,s)=\L_i^{\up}(K,s)+\L_i^{\dw}(K,s)$.

\end{defi}

Let $D_i^s$ be the matrix of $\d_i^s: C^i \to C^{i+1}$ with respect to the bases consisting of elementary cochains.
Then
 $$(D_i^s)_{[\bar{F}]^*, [F]^*}=\sgn([F], \p [\bar{F}])s(F,\bar{F}),$$
The matrix $(D_i^s)^*$ of $ (\d_i^s)^*$ satisfies
$$ ((D_i^s)^*)_{[F]^*, [\bar{F}]^*}=\frac{w(\bar{F})}{w(F)} \sgn([F], \p [\bar{F}])s(F,\bar{F}),$$
and thus
$$ (D_i^s)^*=W_i^{-1} (D_i^s)^\top W_{i+1}.$$
Hence the matrices of $\L_i^{\up}(K,s)$ and $\L_i^{\dw}(K,s)$ are
\begin{equation}\label{LiupS} \L_i^{\up}(K,s)=W_i^{-1} (D_i^s)^\top W_{i+1}D_i^s,
~ \L_i^{\dw}(K,s)=D^s_{i-1} W_{i-1}^{-1}(D^s_{i-1})^\top W_i.
\end{equation}

If $s \equiv 1$, then
$ \L_i^{\up}(K,s)=\L_i^{\up}(K)$ and $\L_i^{\dw}(K,s)=\L_i^{\dw}(K).$
From this perspective, the incidence-signed complex $(K,s)$ generalizes the ordinary complex $K$.
If $s(F,\bar{F}) = \sgn([F], \p [\bar{F}])$ for all pairs $(F,\bar{F})$ with $F \in \p \bar{F}$,
then  $D_i^s = |D_i^s|$, where $|A|:=[|a_{ij}|]$ for a matrix $A=[a_{ij}]$.
In this case, $\L_i^{\up}(K,s)= |\L_i^{\up}(K)|$, and
$\L_i^{\dw}(K,s)= |\L_i^{\dw}(K)|$.

Let $|D_i|$ be the $i$-th incidence matrix of $K$ with rows indexed by $S_{i+1}(K)$ and columns indexed by $S_i(K)$, defined by
$ |D_i|_{\bar{F}, F}=1$ if $F \in \p \bar{F}$,  and $ |D_i|_{\bar{F}, F}=0$ otherwise.
Then
$$ |\L_i^{\up}(K)|=W_i^{-1} |D_i|^\top W_{i+1}|D_i|,$$
which is called the \emph{$i$-up signless Laplace operator} of $K$.
Similarly, the \emph{$i$-down signless Laplace operator} of $K$ is given by
$$ |\L_i^{\dw}(K)|=|D_{i-1}|W_{i-1}^{-1} |D_{i-1}|^\top W_{i}.$$
Notably, when $W_0$ and $W_1$ are both identity matrices, $|\L_0^{\up}(K)|$ reduces to the ordinary \emph{signless Laplacian of a graph} (the $1$-skeleton of $K$) \cite{Hae}, which has been studied in \cite{Des} for graph  nonbipartiteness and surveyed in \cite{Cve}.

In general, consider the chain and cochain groups of $K$ over complex field $\mathbb{C}$.
We assign weights to of each incidence of $K$ from  $\mathbb{C}^*:=\mathbb{C}\backslash \{0\}$.

\begin{defi}
An \emph{incidence-weighted complex} is a pair $(K, \omg)$, where $K$ is a simplicial complex, and $\omg: K \times K \to \mathbb{C}$ such that
$\omg(F, \bar{F}) \in \C^*$ if $F \in \p \bar{F}$, and $\omg(F, \bar{F})=0$ otherwise.
\end{defi}

By a similar discussion, the \emph{weighted boundary map}
$\p_i^\omg: C_i(K,\C) \to C_{i-1}(K,\C)$ is defined by
$$\p_i^\omg[v_0,\ldots,v_i]=\sum_{j=0}^i (-1)^j[v_0,\ldots,\hat{v}_j,\ldots,v_i]
\omg(\{v_0,\ldots,\hat{v}_j,\ldots,v_i\},\{v_0,\ldots,v_i\}).$$
We denote by $\d_i^\omg$ the conjugate of $\p_{i+1}^\omg$, and by $(\d_i^\omg)^*$ the adjoint of $\d_i^\omg$.
The inner product over $\C$ is defined as
$$ \langle f,g \rangle_{C^i}=\sum_{F \in S_i(K)} w(F) f([F])\overline{g([F])},$$
where $\overline{\alpha}$ denotes the conjugate of a complex number $\alpha$.

\begin{defi} Let $(K,\omg)$ be an incidence-weighted complex.

(1) The $i$-up Laplace operator of $(K,\omg)$ is defined as $\L_i^{\up}(K,\omg)=(\d_i^\omg)^* \d_i^\omg$.

(2) The $i$-down Laplace operator of $(K,\omg)$ is defined as $\L_i^{\dw}(K,\omg)=\d_{i-1}^\omg (\d_{i-1}^\omg)^*$.

(3) The $i$-Laplace operator of $(K,\omg)$ is defined as $\L_i(K,\omg)=\L_i^{\up}(K,\omg)+\L_i^{\dw}(K,\omg)$.

\end{defi}

The matrix $D_i^\omg$ of $\d_i^\omg$ is given by
$$(D_i^\omg)_{[\bar{F}]^*, [F]^*}=\sgn([F], \p [\bar{F}])\omg(F,\bar{F}).$$
The matrix $(D_i^\omg)^*$ of $ (\d_i^\omg)^*$ satisfies
$$ ((D_i^\omg)^*)_{[F]^*, [\bar{F}]^*}=\frac{w(\bar{F})}{w(F)} \sgn([F], \p [\bar{F}])\overline{\omg(F,\bar{F})}.$$
Hence, the matrices of $\L_i^{up}(K,\omg)$ and $\L_i^{down}(K,s)$ are
\begin{equation}\label{Liupw} \L_i^{\up}(K,\omg)=W_i^{-1} (D_i^\omg)^\star W_{i+1}D_i^\omg, ~\L_i^{\dw}(K,\omg)=D^\omg_{i-1} W_{i-1}^{-1}(D^\omg_{i-1})^\star W_i.
\end{equation}
where $D^\star$ denotes the conjugate transpose of a complex matrix $D$ to avoid confusion with the adjoint operator $*$.

\section{Covering complexes and incidence graphs}\label{Secov}
\begin{defi}
Let $K,M$ be two complexes. A \emph{simplicial map} from $K$ to $M$ is a map  $\phi: V(K) \to V(M)$ if whenever $\{v_0,\ldots,v_i\} \in K$, then $\{f(v_0), \ldots, f(v_i)\} \in M$. We often use the notation $\phi: K \to M$.
\end{defi}

\begin{defi}\cite{Rot}\label{cov}
Let $K,M$ be  complexes. A pair $(K,\phi)$ is called a \emph{covering complex} of $M$ if the following conditions hold:

(1) $K$ is a connected complex;

(2) $\phi: K \to M$ is a simplicial map;

(3) for each $G \in M$, $\phi^{-1}(G)$ is a union of pairwise disjoint simplices, namely, $ \phi^{-1}(G) =\cup_i F_i$ such that $\phi|_{F_i}: F_i \to G$ is a bijection for each $i$.
\end{defi}

In Definition \ref{cov}, the map $\phi$ is called a \emph{covering map}, $K$ is called a \emph{covering complex}, and $M$ is called the \emph{base  complex} of the covering.
By definition,  $\phi(F) \in S_i(M)$ for every $F \in S_i(K)$.
As noted by Horak and Jost \cite{HJ13A}, the covering complexes are an inaccurate discretization of covering spaces
(in the continuous setting), as they lack a discrete counterpart to the ``locally homeomorphic neighborhood'' requirement of topological covering spaces.
To address this, they introduced the concept a strong covering, which accurately discretizes the continuous notion of a covering.
\emph{We emphasize that all coverings of complexes in this paper are strong coverings.}

\begin{defi}\cite{HJ13B}\label{strcov}
A covering map $\phi: K \to M$ is a \emph{strong covering} if for every $G \in S_i(M)$ which is face of some $\bar{G} \in S_{i+1}(M)$, then for each $F \in \phi^{-1}(G)$, there exists $\bar{F} \in S_{i+1}(K)$ such that $F \in \p \bar{F}$ and $\phi(\bar{F}) = \bar{G}$.
\end{defi}

\begin{lem}\label{localinc}
Let $\phi: K \to M$ be a strong covering.
Then for every $G \in S_i(M)$ and every $F \in \phi^{-1}(G)$,
there exists a bijection also denoted by $\phi$ such that
$$ \phi: \{\bar{F} \in S_{i+1}(K): F \in \p \bar{F}\} \to
\{\bar{G} \in S_{i+1}(M): G \in \p \bar{G}\}, \bar{F} \mapsto \phi(\bar{F}).$$
\end{lem}

\begin{proof}
First, $\phi$ is well-defined, as $G = \phi(F) \subset \phi(\bar{F})$ or $G \in \p \phi(\bar{F})$, meaning $\phi(\bar{F}) \in \{\bar{G} \in S_{i+1}(M): G \in \p \bar{G}\}$.
Next, $\phi$ is an injective map; otherwise, if $\phi(\bar{F})=\phi(\bar{F}'):=\bar{G}$, then $\phi^{-1}(\bar{G})$ would contain
$\bar{F}$ and $\bar{F}'$ with $\bar{F} \cap \bar{F}'=F$, a contradiction to Definition \ref{cov} (3).
Finally, $\phi$ is a surjective map by Definition \ref{strcov}.
\end{proof}

\begin{lem}\cite{HJ13B}\label{covdeg}
Let $\phi: K \to M$ be a strong covering.
There exists a constant $k \in \mathbb{Z}^+$ such that
for each $G \in S_i(M)$ and each $i=0,1,\ldots, \dim K$,
$$ | \{ F \in S_i(K): F \in \phi^{-1}(G)\}| =k.$$
\end{lem}

The constant $k$ in Lemma \ref{covdeg} is called the \emph{degree of the covering}, and
$K$ is accordingly termed a \emph{$k$-fold covering} of $M$.

We now turn to the graph coverings.
Let $G$ be a simple graph with vertex set $V(G)$ and edge set $E(G)$.
The neighborhood of a vertex $v \in V(G)$ is denoted and defined by
$N_G(v)=\{u: \{u,v\} \in E(G)\}$.

\begin{defi}
Let $\bar{G}$ and $G$ be simple graphs, where $G$ is connected.
A surjective map $\phi: V(\bar{G}) \to V(G)$ is called a \emph{covering map} if

(1)  $\phi$ is a homomorphism,  namely, $\phi(e) \in E(G)$ for each $e \in E(\bar{G})$,

(2) for each vertex $v \in V(G)$ and each vertex $\bar{v} \in \phi^{-1}(v)$, the restriction $\phi|_{N_{\bar{G}}(\bar{v})}: N_{\bar{G}}(\bar{v}) \to N_G(v)$ is a bijection.
\end{defi}

This definition is standard in literature \cite{Ami, Sta}.
Notably, a graph covering (with $\bar{G}$ connected) coincides with the strong covering of the $1$-skeleton of a complex.
 Gross and Tucker \cite{GT} established a relationship between $k$-fold coverings and derived graphs (see Section \ref{drv} for the notion of derived graphs).

\begin{lem}\cite{GT}\label{per}
Let $\bar{G}$ be a $k$-fold covering of a graph $G$. Then there exists an assignment $\psi$ of permutations in $\mathbb{S}_k$ to the edges of $G$ such that $G^\psi$ is isomorphic to $\bar{G}$.
\end{lem}

\begin{rmk}\label{rmkper}
In Lemma \ref{per}, let $\eta: V(G^\psi) \to V(\bar{G})$ be the isomorphism, $\phi: V(\bar{G}) \to V(G)$  the covering map,
and $p: V(G^\psi) \to V(G)$ the natural projection with $p(v,i)=v$ for all $i \in [k]$.
By the proof of Lemma \ref{per} (see \cite[Theorem 2]{GT}), we have
$$ \phi \eta = p,$$
or equivalently, for all $v \in V(G)$ and $i \in [k]$,
$$ \phi \eta (v,i)= v.$$
\end{rmk}

\begin{proof}[Proof of Lemma \ref{perCom}]
By the definition of strong covering,
$\phi$ maps $S_i(K)$ to $S_i(M)$ and $S_{i+1}(K)$ to $S_{i+1}(M)$.
Consequently, $\phi$ induces a surjective $k$-to-$1$ map (also denoted by $\phi$) from the vertex set of $B_i(K)$ to the vertex set of $B_i(M)$.
If $F \in \p \bar{F}$, then $\phi(F) \in \p \phi(\bar{F})$,
meaning $\phi$ maps the edges of $B_i(K)$ to the edges of $B_i(M)$.
By Lemma \ref{localinc}, for each $G \in S_i(M)$ and $F \in \phi^{-1}(G)$,  the restriction $\phi|_{N_{B_i(K)}(F)}: N_{B_i(K)}(F) \to N_{B_i(M)}(G)$ is bijective.
For each $\bar{G} \in S_{i+1}(M)$ and $\bar{F} \in \phi^{-1}(\bar{G})$,
note that $N_{B_i(K)}(\bar{F})=\p \bar{F}$ and $N_{B_i(M)}(\bar{G})=\p \bar{G}$.
Since $\phi|_{\bar{F}}: \bar{F} \to \bar{G}$ is a bijection,
the restriction $\phi|_{N_{B_i(K)}(\bar{F})}: N_{B_i(K)}(\bar{F}) \to N_{B_i(M)}(\bar{G})$  is also a bijection.
Thus, $\phi$ induces a  $k$-fold covering from $B_i(K)$ to $B_i(M)$.

By Lemma \ref{per},
there exists a mapping $\psi: E(B_i(M)) \to \mathbb{S}_k$ such that $B_i(M)^\psi$ is isomorphic to $B_i(K)$.
Let $\eta: B_i(M)^\psi \to B_i(K)$ be the isomorphism.
For any $(F,i) \in S_i(M) \times [k]$,
$\eta(F,i) \in S_i(K)$; otherwise, if $\eta(F,i) \in S_{i+1}(K)$,
then $\phi \eta(F,i) \in S_{i+1}(M)$, a contradiction to
Remark \ref{rmkper} (since $\phi \eta(F,i) = p(F,i) =F \in S_i(M)$).
So $\eta$ sends $S_i(M) \times [k]$ to $S_i(K)$.
A similar argument shows that $\eta$ maps $S_{i+1}(M) \times [k]$ to $S_{i+1}(K)$.
\end{proof}

We give an example to illustrate Lemma \ref{perCom} as follows.

\begin{exm}\label{exmcov}
Let $K,M$ be the complexes depicted in Fig. \ref{KL}, where the vertex set of each triangle represents a $2$-simplex.
It is easy to verify that  $K$ is a $2$-fold covering of $M$.

\begin{figure}[h]
\centering
\includegraphics[scale=.75]{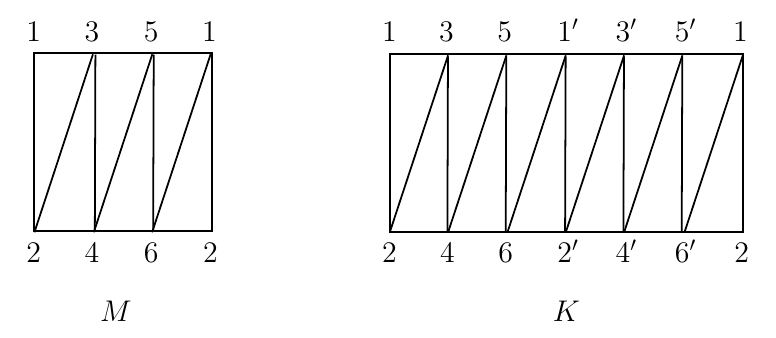}~~~
\vspace{-3mm}
\caption{\small A complex $M$ and its $2$-fold covering $K$}\label{KL}
\end{figure}

\begin{figure}[h]
\centering
\includegraphics[scale=.8]{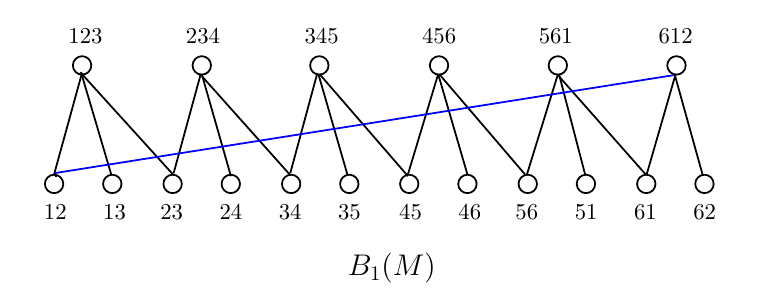}~~~
\vspace{-3mm}
\caption{\small The incidence graph $B_1(M)$}\label{B1L}
\end{figure}

\begin{figure}[h]
\centering
\includegraphics[scale=.8]{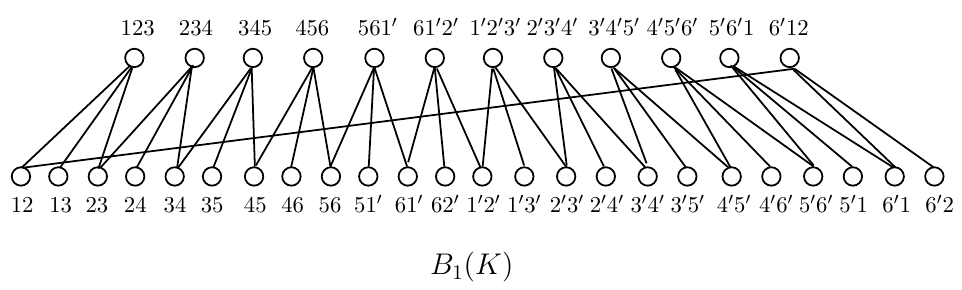}~~~
\vspace{-3mm}
\caption{\small The incidence graph $B_1(K)$}\label{B1K}
\end{figure}

The incidence graphs $B_1(M)$ and $B_1(K)$ are shown in Fig. \ref{B1L} and
Fig. \ref{B1K} respectively. Here, an edge $\{a,b\}$ is abbreviated as $ab$ and a face $\{a,b,c\}$ is abbreviated as $abc$.
Define a permutation assignment $\psi: E(B_1(M)) \to \mathbb{S}_2$ such that $\psi(12, 612)=(1\ 2)$ (see the blue edge in Fig. \ref{B1L}) and $\phi(F,\bar{F})=1$ for all other incidences $(F, \bar{F})$ with $F \in \p \bar{F}$.
The derived graph $B_1(M)^\psi$ is displayed in Fig. \ref{B1Lpsi}, where hollow circles represent vertices $(abc,1)$ or $(ab,1)$, and solid circled represent vertices $(def,2)$ or $(de,2)$.
It is readily seen there exists an isomorphism $\eta: B_1(M)^\psi \to B_1(K)$ which sends $S_1(M) \times \{1,2\}$ to $S_1(K)$ and $S_2(M) \times \{1,2\}$ to $S_2(K)$.

\begin{figure}[h]
\centering
\includegraphics[scale=.8]{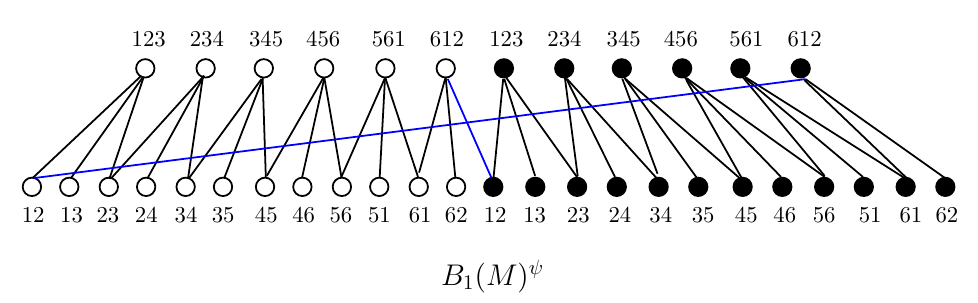}~~~
\vspace{-3mm}
\caption{\small The derived graph $B_1(M)^\psi$}\label{B1Lpsi}
\end{figure}

\end{exm}

\section{Spectra of covering complexes}\label{sec4}
In this section we provide the proof of the main results stated in Section \ref{main}.
We also establish the counterpart results for the $i$-down Laplace operators and
consider the special case that the group generated by permutation assignment is abelian.

Let $\phi: K \to M$ be a $k$-fold covering.
By Lemma \ref{perCom}, $B_i(K)$ is a $k$-fold covering of $B_i(M)$,
and there exists a permutation assignment $\psi$ on $B_i(M)$ such that $B_i(M)^\psi$ is isomorphic to $B_i(K)$ via a mapping $\eta$ that sends $S_i(M) \times [k]$ to $S_i(K)$ and $S_{i+1}(M) \times [k]$ to $S_{i+1}(K)$.

Recall that $\L_i^{\up}(K)=W_i(K)^{-1} D_i(K)^\top W_{i+1}(K) D_i(K)$.
By the isomorphism $\eta: B_i(M)^\psi \to B_i(K)$, as discussed in Section \ref{Rcob}, we can relabel the rows and columns of $D_i(K)$, $W_i(K)$ and $W_{i+1}(K)$ such that
\begin{equation}\label{DiLp}
\begin{split} & D_i(K)_{(\bar{G},l),(G,j)}
=\sgn([\eta(G,j)], \p [\eta(\bar{G},l)]),\\
& W_i(K)_{(G,j),(G,j)}=w(\eta(G,j)),
W_{i+1}(K)_{(\bar{G},l),(\bar{G},l)}=w(\eta(\bar{G},l)).
\end{split}
\end{equation}

We require  a result on the sign relationship under covering map $\phi$.
Let $F, \bar{F} \in K$ with $F \in \p \bar{F}$.
By Formula (3.5) in \cite{HJ13B},
\begin{equation}\label{sgncov} \sgn([\phi(F)], \p [\phi(\bar{F}))=
\sgn([F], \p [\bar{F}]) \sgn([F], [\phi(F)]) \sgn ([\bar{F}], [\phi(\bar{F})]),
\end{equation}
where, if $[F]=[v_0,\ldots,v_i]$, then
\begin{equation}\label{sgn2}
\sgn([F], [\phi(F)]): =
\begin{cases} 1,& \text{if~} [\phi(v_0),\ldots,\phi(v_i)] \text{~is positively oriented}, \\
-1, & \text{otherwise}.
\end{cases}
\end{equation}

\begin{proof}[Proof of Lemma \ref{F-DiK}]
By \eqref{DiLp}, $D_i(K)_{(\bar{G},l),(G,j)}=\sgn([\eta(G,j)], \p [\eta(\bar{G},l)])$.
Note that $\eta(G,j) \in \p \eta(\bar{G},l)$ (equivalently $\{\eta(G,j),\eta(\bar{G},l)\} \in E(B_i(K))$ by definition) if and only if
$\{(G,j), (\bar{G},l)\} \in E(B_i(M)^\psi)$ (by the isomorphism $\eta$), which holds if and only if $G \in \p \bar{G}$ and
$l=\psi(\bar{G}, G)(j)$ by definition.
Thus, by  (\ref{sgncov}), we have
\begin{equation}\label{signrel}
\begin{split}
\sgn([\eta(G,j)], \p [\eta(\bar{G},l)]) & = \sgn ([\phi\eta(G,j)], \p [\phi\eta(\bar{G},l)]) \sgn([\eta(G,j)], [\phi\eta(G,j)]) \\
& ~~~ \cdot \sgn([\eta(\bar{G},l)], [\phi\eta(\bar{G},l)])\\
&= \sgn([G], \p [\bar{G}]) \sgn([\eta(G,j)], [G])
 \sgn([\eta(\bar{G},l)], [\bar{G}]),
 \end{split}
\end{equation}
where the second equality follows from Remark \ref{rmkper}, which gives
$\phi\eta(\bar{G},l)=p(\bar{G},l)=\bar{G}$ and $\phi\eta({G},j)=p({G},j)=G$.

Let $\Lambda_{i+1}(M)$ be defined as in \eqref{Lmdi1p}, i.e.,
$$\Lambda_{i+1}(M)^\eta_{(\bar{G},l),(\bar{G},l)}=\sgn([\eta(\bar{G},l)], [\bar{G}]),$$
and let $\Lambda_i(M)^\eta$ be defined as in \eqref{Lmdip}, i.e.,
$$ \Lambda_i(M)^\eta_{(G,j),(G,j)}=\sgn([\eta(G,j)], [G]). $$
Let $ D_i(M)^\psi$ be defined as in \eqref{tdilp}, i.e.,
$$
 D_i(M)^\psi_{(\bar{G},l),(G,j)}=\left\{
\begin{array}{ll}
\sgn([G], \p [\bar{G}]), & \mbox{if~} G \in \p \bar{G}, l = \psi(\bar{G},G)(j),\\
0, & \mbox{else}.
\end{array}
\right.
$$
Then we conclude that
$$ D_i(K)=\Lambda_{i+1}(M)^\eta D_i(M)^\psi \Lambda_i(M)^\eta.$$
\end{proof}

\subsection{Proof of main results}
\begin{proof}[Proof of Theorem \ref{mainthm}]
By  Lemma \ref{F-DiK}, we have
\begin{align*}
 \L^{\up}_i(K)& = W_i(K)^{-1}  D_i(K)^\top W_{i+1}(K) D_i(K)\\
 &= W_i(K)^{-1} (\Lambda_i(M)^\eta)^\top (D_i(M)^\psi)^\top (\Lambda_{i+1}(M)^\eta)^\top W_{i+1}(K)  \Lambda_{i+1}(M)^\eta D_i(M)^\psi \Lambda_i(M)^\eta.
\end{align*}
Thus, the matrix
$\tilde{\L}^{\up}_i(K):=\Lambda_i(M)^\eta \L^{up}_i(K) (\Lambda_i(M)^\eta)^{-1}$ takes the form:
$$ \tilde{\L}^{\up}_i(K)=\Lambda_i(M)^\eta W_i(K)^{-1} (\Lambda_i(M)^\eta)^\top (D_i(M)^\psi)^\top (\Lambda_{i+1}(M)^\eta)^\top W_{i+1}(K) \Lambda_{i+1}(M)^\eta D_i(M)^\psi.$$
Since $\Lambda_i(M)^\eta,\Lambda_{i+1}(M)^\eta$ are signature matrices (diagonal matrices with $\pm 1$ on the diagonal) and $W_i(K),W_{i+1}(K)$ are diagonal matrices, we have
$$
\tilde{\L}_i^{\up}(K)  = W_i(K)^{-1} (D_i(M)^\psi)^\top  W_{i+1}(K) D_i(M)^\psi.$$

By (\ref{DiLp}), the weight assumption and Remark \ref{rmkper},  $$W_i(K)_{(G,j),(G,j)}=w(\eta(G,j))=w(\phi\eta(G,j))=w(G),$$
and similarly $W_{i+1}(K)_{(\bar{G},l),(\bar{G},l)}=w(\bar{G})$.
So
\begin{equation}\label{kron}
W_i(K)=W_i(M) \otimes I_k, ~ W_{i+1}(K)=W_{i+1}(M) \otimes I_k.
\end{equation}

Let $\Psi_i$ be the subgroup of $\mathbb{S}_k$ generated by the permutations assigned by $\psi$ as defined in \eqref{Psi}.
Let $\varrho: \Psi_i \to \text{GL}_k(\mathbb{C})$ denote the permutation representation $g \mapsto P^g=(p_{ij}^g)$ (defined in \eqref{PRe}).
For matrices $D^g_i(M)$ as in \eqref{Dg}, we recall the decomposition (\ref{Dsum}):
\begin{equation}\label{Dsum2}
 D_i(M)^\psi=\sum_{g \in \Psi_i} D_i^g(M) \otimes P^g.
 \end{equation}

Combining  (\ref{kron}) and \eqref{Dsum2}, we have
\begin{align*}
\tilde{\L}_i^{\up}(K)& =(W_i(M)^{-1} \otimes I_k )\bigg(\sum_{g \in \Psi_i} D_i^g(M)^\top \otimes P_g^\top\bigg)  (W_{i+1}(M) \otimes I_k) \bigg(\sum_{g \in \Psi_i}D_i^g(M) \otimes P_g \bigg)\\
& = \sum_{g,g'\in \Psi_i} \left(W_i(M)^{-1} D_i^g(M)^\top W_{i+1}(M) D_i^{g'}(M)\right) \otimes \left((P^g)^\top P^{g'}\right).
\end{align*}
Let $\varrho_1,\varrho_2,\ldots,\varrho_t$ be all irreducible sub-representations of $\varrho$, where $\varrho_1=1$ is the identity representation of degree one.
By (\ref{rep}), there exists an invertible matrix  $T$ such that for all $g \in \Psi_i$,
$$T^{-1} P^g T= I_1 \oplus \varrho_2(g) \oplus \cdots \oplus \varrho_t(g).
$$
Note that $(P^g)^\top = P^{g^{-1}}$.
So we have
\begin{align*}
(I \otimes T)^{-1} \tilde{\L}^{\up}(K) (I \otimes T)
&= \sum_{g,g'\in \Psi_i} \left(W_i(M)^{-1} D_i^g(M)^\top W_{i+1}(M) D_i^{g'}(M)\right) \otimes \left(T^{-1}(P^g)^\top P^{g'}T\right)\\
 &= \sum_{g,g'\in \Psi_i} \left(W_i(M)^{-1} D_i^g(M)^\top W_{i+1}(M) D_i^{g'}(M)\right) \otimes
  \left(I_1 \oplus \bigoplus_{j=2}^t \varrho_j(g^{-1})\varrho_j(g')\right)\\
&=\sum_{g,g'\in \Psi_i} \left(W_i(M)^{-1} D_i^g(M)^\top W_{i+1}(M) D_i^{g'}(M)\right)\\
& ~~~ \oplus   \sum_{g,g'\in \Psi_i} \left(W_i(M)^{-1} D_i^g(M)^\top W_{i+1}(M) D_i^{g'}(M)\right)
 \otimes \left(\bigoplus_{j=2}^t \varrho_j(g^{-1})\varrho_j(g')\right)\\
 &= W_i(M)^{-1} \bigg(\sum_{g\in \Psi_i} D_i^g(M)\bigg)^\top W_{i+1}(M) \bigg(\sum_{g\in \Psi_i} D_i^g(M) \bigg) \\
& ~~~ \oplus \bigoplus_{j=2}^t \sum_{g,g'\in \Psi_i} \left(W_i(M)^{-1} D_i^g(M)^\top W_{i+1}(M) D_i^{g'}(M)\right) \otimes \varrho_j(g^{-1})\varrho_j(g').
\end{align*}
Set $X:=(I \otimes T)^{-1} \tilde{\L}^{up}(K) (I \otimes T)$ temporarily.
Note that $D_i(M)=\sum_{g \in \Psi_i} D_i^g(M)$ by (\ref{decompL}).
So we have
\begin{equation}\label{decomp2}
\begin{split}
X &=
  W_i(M)^{-1} D_i(M)^\top W_{i+1}(M) D_i(M) \\
& ~~~ \oplus \bigoplus_{j=2}^t  \bigg(\! \sum_{g\in \Psi_i} (W_i(M)^{-1} D_i^g(M)^\top) \otimes \varrho_j(g^{-1}) \!\bigg)\! \bigg(\!\sum_{g\in \Psi_i} (W_{i+1}(M)D_i^g(M)) \otimes \varrho_j(g)\! \bigg)\\
& = \L_i^{\up}(M) \\
& ~~~\oplus \bigoplus_{j=2}^t \!\bigg(\! \sum_{g\in \Psi_i} (W_i(M)^{-1} D_i^g(M)^\top) \otimes \varrho_j(g^{-1}) \!\bigg) \! \bigg(\!\sum_{g\in \Psi_i} (W_{i+1}(M)D_i^g(M)) \otimes \varrho_j(g)\!\bigg).
\end{split}
\end{equation}
This completes the proof.
\end{proof}

By Theorem \ref{mainthm}, if we know all irreducible sub-representations $\varrho_j$ of $\Psi_i$, then we will get all remaining eigenvalues of the covering complex $K$.
In particular, we get the spectral inclusion theorem (Theorem \ref{inclusion}) established by Horak and Jost \cite{HJ13B}.

\begin{proof}[Proof of Theorem \ref{inclusion}]
For the combinatorial Laplace operators $L_i^{up}(M)$ and $L_i^{up}(K)$, the weights satisfy  $w_K(F)=w_M(\phi(F))=1$ for all $F \in K$.
For the normalized Laplace operators $\Delta_i^{up}(M)$ and $\Delta_i^{up}(K)$,
by Lemma \ref{localinc},
for each $F \in K$, if $F$ is a facet, so is $\phi(F)$, and  $w_K(F)=w_M(\phi(F))=1$;
for non-facets $F$, we still have $w_K(F)=w_M(\phi(F))$ by induction.
So the result follows by Theorem \ref{mainthm}.
\end{proof}

Next, by Theorem \ref{mainthm}, we get the spectral union property for $2$-fold covering complexes,  which extends Bilu and Linial's work \cite{Bilu} from graphs to complexes.

\begin{proof}[Proof of Theorem \ref{SpecUnion}]
By Lemma \ref{perCom}, there exists a permutation assignment $\psi:E(B_i(M)) \to \mathbb{S}_2$ such that $B_i(M)^\psi$ is isomorphic to $B_i(K)$.
Let $\Psi_i=\langle \text{image}(\psi) \rangle \le \mathbb{S}_2$ as defined in (\ref{Psi}).
As $K$ is connected, so is $B_i(K)$ or $B_i(M)^\psi$.
If $\Psi_i=1$, by definition $B_i(M)^\psi$ is a union of two disjoint copies of $B_i(M)$; a contradiction.
Thus, $\Psi_i=\mathbb{S}_2$.

It is known $\mathbb{S}_2$ has two irreducible representations: the trivial representation $\varrho_1=1$ and the sign representation $\varrho_2=\sgn$.
There exists an invertible matrix $T$ such that for each $g \in \mathbb{S}_2$,
$$ T^{-1} P^g T =I_1 \oplus \sgn g.$$
So, by Eq. (\ref{decomp2}), noting that $\sgn g^{-1} =\sgn g$,
$$(I \otimes T)^{-1} \tilde{\L}^{\up}(K) (I \otimes T)=
 \L_i^{\up}(M) \oplus W_i(M)^{-1} \bigg(\sum_{g\in \Psi_i} D_i^g(M) \sgn g \bigg)^\top W_{i+1}(M) \bigg(\sum_{g\in \Psi_i} D_i^g(M) \sgn g\bigg).$$

Define an incidence-signed complex $(M,s)$ with $s(F, \bar{F})=\sgn \psi(F, \bar{F})$ for $F \in \p \bar{F}$.
 Then, similar to the discussion for (\ref{Dsum}), we get the matrix $D_i^s(M)$ of $\d_i^s$ on $(M,s)$, namely,
 $$D_i^s(M)= \sum_{g \in \Psi_i} D_i^g(M) \sgn g.$$
 So, by Eq. (\ref{LiupS}),
$$
(I \otimes T)^{-1} \tilde{\L}^{\up}(K) (I \otimes T)=
 \L_i^{\up}(M) \oplus W_i(M)^{-1} (D_i^s(M))^\top W_{i+1}(M) D_i^s(M)
 = \L_i^{\up}(M) \oplus  \L_i^{\up}(M,s).
$$
The result follows since $ \tilde{\L}^{\up}(K)=\Lambda_i(M)^\eta \L^{\up}_i(K) (\Lambda_i(M)^\eta)^{-1}$ as denoted in the proof of Theorem \ref{mainthm}.
\end{proof}

We give a remark for Theorem \ref{SpecUnion} on the connectedness of $K$.
By the definition of strong covering $\phi: K \to M$, $K$ is required to be connected, which forces $M$ must be connected.
If we remove the condition of $K$ being connected, then $\Psi_i$ may be trivial ($\Psi_i=\{1\}$).
Then $\sgn g =1$ for all  $g \in \Psi_i$.
So, $(M,s)=M$, and
$\spec \L_i^{\up}(K) = \spec \L_i^{\up}(M) \cup \spec \L_i^{\up}(M)$, preserving the union property.

Below, we provide an example to illustrate Theorem \ref{SpecUnion}.

\begin{exm}
Let $K,M$ be the complexes in Fig. \ref{KL}.
The spectrum of $L_1^{\up}(M)$ is
$$ \spec L_1^{\up}(M)=\{5, 4^{(2)}, 2^{(2)}, 1, 0^{(6)}\},$$
and the spectrum of $L_1^{\up}(K)$ is
$$ \spec L_1^{\up}(K)=\{5,4^{(2)},3^{(2)},2^{(2)},1, (3+\sqrt{3})^{(2)}, (3-\sqrt{3})^{(2)},0^{(12)}\},$$
where $\la^{(m)}$ denotes eigenvalue $\la$ with multiplicity $m$.

From Example \ref{exmcov}, let $\psi: E(B_1(M)) \to \mathbb{S}_2$ be such that $\psi(12, 612)=(1\ 2)$  and $\phi(F,\bar{F})=1$ for all other incidences $(F, \bar{F})$ with $F \in \p \bar{F}$.
Then $B_1(M)^\psi$ is isomorphic to $B_1(K)$.
Define an incidence-signed complex $(M,s)$ such that $s(F, \bar{F})=\sgn \psi(F, \bar{F})$ for each incidence $(F,  \bar{F})$ with $F \in \p \bar{F}$.
So we have $s(12, 612)=-1$ and $s(F,  \bar{F})=1$ for other incidences $(F,  \bar{F})$.
The matrix $D_1^s$ of $\d_1^s$ on $(M,s)$ is obtained from the matrix $D_1$ of $\d_1$ on $M$ only by replacing the $([612]^*,[12]^*)$-entry $1$ by $-1$.
The spectrum of $L_1^{\up}(M,s)$ is
$$\spec L_1^{\up}(M,s)=\{3^{(2)}, (3+\sqrt{3})^{(2)}, (3-\sqrt{3})^{(2)},0^{(6)}\}.$$
So, as multiset union,
$$ \spec L_1^{\up}(K)=\spec L_1^{\up}(M) \cup \spec L_1^{\up}(M,s).$$
\end{exm}

\subsection{Abelian group case}
We consider a special case where the group $\Psi_i$ defined in (\ref{Psi}) is  abelian.
This can be considered as a generalization of Theorem \ref{SpecUnion} as $\mathbb{S}_2$ is abelian.

\begin{thm}\label{Abelcov}
Let $\phi: K \to M$ be a $k$-fold covering map, and  let $\psi: E(B_i(M) \to \mathbb{S}_k$ be the permutation assignment such that $B_i(M)^\psi$ is isomorphic to $B_i(K)$.
Suppose that $w_K(F)=w_M(\phi(F))$ for all $F \in K$.
Let $\Psi_i$ be the group defined in (\ref{Psi}), which is abelian.
Then $\Psi_i$ has exactly $k$ elements, and the permutation representation $\varrho$ of $\Psi$ decomposes into irreducible sub-representations of degree one as follows:
\begin{equation}\label{repabl} \varrho=I_1 \oplus \varrho_2 \oplus \cdots \oplus \varrho_k,
\end{equation}
and
$$ \spec \L_i^{\up}(K))= \spec\L_i^{\up}(M) \cup \bigcup_{j=2}^k \spec\L_i^{\up}(M,\omg_j),$$
where $\omg_j(F,\bar{F})=\varrho_j(\psi(F,\bar{F}))$ for
 $(F,  \bar{F}) \in S_i(M) \times S_{i+1}(M)$ such that $F \in \p \bar{F}$.
\end{thm}

\begin{proof}
As $K$ is connected, so is $B_i(K)$ (and hence $B_i(M)^\psi$).
By Theorem 3.4 of \cite{Song}, $\Psi_i$ acts transitive on the set $[k]$ via  permutations because $B_i(M)^\psi$ is connected.
We claim that $\Psi_i$ acts faithfully on $[k]$.
Suppose there exists $g \in \Psi_i$ and $j \in [k]$ such that $g(j)=j$.
For any $\ell \in [k]$, transitivity implies that $\ell=h(j)$ for some $h \in \Psi_i$.
Since $\Psi_i$ is Abelian,
$$ g(\ell)=g  h(j)=hg(j)=h(j)=\ell,$$
which implies $g=1$.
By the orbit-stabilizer lemma, $\Psi_i$ has exactly $k$ elements.
As abelian groups has only degree-one irreducible representations,
we immediately obtain the decomposition (\ref{repabl}).
Moreover, there exists an invertible matrix $T$ such that
$$ T^{-1} P_g T=I_1 \oplus \varrho_2(g) \oplus \cdots \oplus \varrho_k(g).$$

By (\ref{decomp2}), noting that each $\varrho_j$ has degree one, we have
\begin{align*}
(I \otimes T)^{-1} \tilde{\L}^{\up}(K) (I \otimes T)
&=\L_i^{\up}(M) \oplus \bigoplus_{j=2}^k\bigg(\! \sum_{g\in \Psi_i} \! W_i(M)^{-1} D_i^g(M)^\top\varrho_j(g^{-1})\!\bigg) \! \bigg(\! \sum_{g\in \Psi_i}\! W_{i+1}(M) D_i^g(M)\varrho_j(g)\!\bigg)\\
&=\L_i^{\up}(M) \oplus \bigoplus_{j=2}^k W_i(M)^{-1} \bigg( \sum_{g\in \Psi_i} D_i^g(M)^\top\overline{\varrho_j(g)}\bigg) W_{i+1}(M) \bigg( \sum_{g\in \Psi_i} D_i^g(M)\varrho_j(g)\bigg)\\
&=\L_i^{\up}(M) \oplus \bigoplus_{j=2}^k W_i(M)^{-1} \bigg(\sum_{g\in \Psi_i} D_i^g(M)\varrho_j(g)\!\bigg)^\star W_{i+1}(M) \bigg(\sum_{g\in \Psi_i} D_i^g(M)\varrho_j(g)\!\bigg).
\end{align*}

For $j=2,\ldots,k$, let $(M, \omg_j)$ be an incidence-weighted complex such that $\omg_j(F,\bar{F})=\varrho_j(\psi(F,\bar{F}))$ for $F \in \p \bar{F}$.
Then, similar to the derivation of (\ref{Dsum}), the matrix $D_i^{\omg_j}$ of $\d_i^{\omg_j}$ on $(M, \omg_j)$ is precisely
$D^{\omg_j}(M)=\sum_{g \in \Psi_i} D_i^g(M) \varrho_j(g)$.
 By Eq. (\ref{Liupw}), we have
 we have
  $$ (I \otimes T)^{-1} \tilde{\L}^{\up}(K) (I \otimes T) = \L_i^{\up}(M) \oplus \bigoplus_{j=2}^t \L_i^{\up}(M,\omg_j).$$
  The result now follows.
\end{proof}

A special case in Theorem \ref{Abelcov} occurs when $\Psi_i=\langle g \rangle$, where $g=(a_1 a_2 \ldots a_k)$ is a cycle in $\mathbb{S}_k$.
In this case, the irreducible representations of $\Psi_i$ are readily determined, namely,
$\varrho_j(g)=e^{\textbf{i}\frac{ 2 \pi j}{k}}$ for $j=0,1,\ldots,k-1$, where $\textbf{i}=\sqrt{-1}$.

\subsection{A decomposition of $\L_i^{\dw}(K)$}
By analogous reasoning to the preceding section, we obtain results for the $i$-down Laplace operator of the covering  complex $K$.
Let $\phi: K \to M$ be a $k$-fold covering.
By Lemma \ref{perCom}, $B_{i-1}(K)$ is a $k$-fold covering of $B_{i-1}(M)$,
and there exists a permutation assignment $\psi$ on $B_{i-1}(M)$ such that $B_{i-1}(M)^\psi \cong B_{i-1}(K)$ via an isomorphism $\eta$ that maps $S_{i-1}(M) \times [k]$ to $S_{i-1}(K)$ and $S_i(M) \times [k]$ to $S_i(K)$.
Assume weights satisfy $w_K(F)=w_M(\phi(F))$ for all $F \in S_{i-1}(K) \cup S_i(K)$.
Define the subgroup:
\begin{equation}\label{Psi2}\Psi_{i-1}=\langle \psi(G,\bar{G}): G \in S_{i-1}(M), \bar{G} \in S_i(M), G \in \p \bar{G}\rangle \le \mathbb{S}_k,
\end{equation}
and lets its permutation representation $\varrho$ decompose into irreducible sub-representations as:
$$ \varrho=1 \oplus \varrho_2 \oplus \cdots \oplus \varrho_t,$$
where $1$ denotes the trivial representation.
For each $g \in \Psi_{i-1}$, let $D^g_{i-1}(M)=(D^g_{\bar{G}, G})$ be defined as in (\ref{Dg}).
Then there exists an invertible matrix $T$ such that
\begin{multline*}
(I \otimes T)^{-1} (\Lambda_i(M)^\eta)^{-1}\L_i^{\dw}(K)\Lambda_i(M)^\eta (I \otimes T) \\
=\L_i^{\dw}(M) \oplus \bigoplus_{j=2}^t \bigg( \sum_{g\in \Psi_{i-1}} (D_{i-1}^g(M) W_{i-1}(M)^{-1} ) \otimes \varrho_j(g) \bigg) \bigg(\sum_{g\in \Psi_{i-1}} (D_{i-1}^g(M)^\top W_i(M)) \otimes \varrho_j(g^{-1})\bigg),
\end{multline*}
where  $\Lambda_i(M)^\eta$ is  defined  as in (\ref{Lmdip}).

\begin{thm}
Let $\phi: K \to M$ be a $k$-fold covering map.
Then
$$ \spec L_i^{\dw}(M) \subset \spec L_i^{\dw}(K), ~ \spec \Delta_i^{\dw}(M) \subset \spec \Delta_i^{\dw}(K).$$
\end{thm}

\begin{thm}
Let $\phi: K \to M$ be a $2$-fold covering map.
Suppose that $w_K(F)=w_M(\phi(F))$ for all $F \in S_{i-1}(K) \cup S_i(K)$.
Then
$$ \spec \L_i^{\dw}(K))= \spec\L_i^{\dw}(M) \cup \spec\L_i^{\dw}(M,s),$$
where $s(F, \bar{F})=\sgn \psi(F, \bar{F})$ for $(F,  \bar{F}) \in S_{i-1}(M) \times S_i(M)$ with $F \in \p \bar{F}$, and $\psi:E(B_{i-1}(M)) \to \mathbb{S}_2$ such that $B_{i-1}(M)^\psi \cong B_{i-1}(K)$.
\end{thm}

\begin{thm}
Let $\phi: K \to M$ be a $k$-fold covering map, let $\Psi_{i-1}$ be the group defined in (\ref{Psi2}) that is abelian.
Suppose that $w_K(F)=w_M(\phi(F))$ for all $F \in S_{i-1}(K) \cup S_i(K)$.
Then
$$\spec \L_i^{\dw}(K))= \spec\L_i^{\dw}(M) \cup \bigcup_{j=2}^k \spec\L_i^{\dw}(M,\omg_j),$$
where
the permutation representation $\varrho$ of $\Psi_{i-1}$ has a decomposition
$\varrho=1 \oplus \varrho_2 \oplus \cdots \oplus \varrho_k$ of irreducible sub-representation of degree one,
and
$\omg_j(F,\bar{F})=\varrho_j(\psi(F,\bar{F})$ for
 $(F,  \bar{F}) \in S_{i-1}(M) \times S_i(M)$ such that $F \in \p \bar{F}$.
\end{thm}

\section{Cohomology groups of covering complexes}
Eckmann \cite{Eckmann} (also see \cite{HJ13B}),
showed that the $i$-th reduced cohomology group of a complex $K$ is isomophic to the kernel of the $i$-Laplace operator $\L_i(K)$, namely,
$ \tilde{H}^i(K,\mathbb{R})= \text{ker}\L_i(K)$.
The \emph{$i$-th Betti number} of $K$, denoted by $\beta_i(K)$, is defined be the dimension of $\tilde{H}^i(K,\R)$.
In this section, we prove that if $K$ is a covering complex over $M$,
then $\beta_i(K) \ge \beta_i(M)$.

Horak and Jost \cite{HJ13B} provided explicit formulas for $i$-up and $i$-down Laplace operators as follows.
\begin{multline}\label{kerup}
 (\L_i^{up}(M) f)([G])  =\!\!\!\sum_{\bar{G} \in S_{i+1}(M): \atop G \in \p \bar{G}}\!\! \frac{w(\bar{G})}{w(G)}f([G])+ \!\!\!\!\!\!\!\!\!\sum_{G' \in S_i(M):\atop G' \cup G = \bar{G} \in S_{i+1}(M)}\!\!\!\!\!\!\!\!\!\!\frac{w(\bar{G})}{w(G)}\sgn([G], \p [\bar{G}]) \sgn([G'], \p [\bar{G}])f[G'],\\
(\L^{down}_i(M)f)([G]) =\!\!\!\!\!\sum_{H \in S_{i-1}(M):\atop H \in \p G} \!\!\!\frac{w(G)}{w(H)} f([G])+\!\!\!\!\!\!\!\!\!
\sum_{G'\in S_i(M): \atop G\cap G'=H \in S_{i-1}(M)} \!\!\!\! \!\!\!\!\!\!\!\frac{w(G')}{w(H)}\sgn([H],\p [G])\sgn([H], \p [G']) f([G']).
\end{multline}

\begin{lem}\label{kerupdown}
Let $\phi: K \to M$ be a $k$-fold covering map.
Suppose that $\frac{w_K(\bar{F})}{w_K(F)}=\frac{w_M(\phi(\bar{F}))}{w_M(\phi(F))}$ for each pair $(F,\bar{F}) \in S_i(K) \times S_{i+1}(K)$ with $F \in \p \bar{F}$.
For $f \in C^i(M,\R)$, define $\bar{f} \in C^i(K,\R)$ by
\begin{equation}\label{barf} \bar{f}([F])=f([\phi(F))])\sgn([F], [\phi(F)]).\end{equation}
Then the following results hold.

$(1)$  If $f \in \text{ker}\L^{\up}_i(M)$, then $\bar{f} \in \text{ker}\L^{\up}_i(K)$.

$(2)$  If $f \in \text{ker}\L^{\dw}_i(M)$, then $\bar{f} \in
\text{ker}\L^{\dw}_i(K)$.
\end{lem}

\begin{proof}
(1) If $f  \in \text{ker}\L^{\up}_i(M)$,
then for each $G\in S_i(M)$,  $(\L_i^{up}(M) f)([G])=0$.
By Lemma \ref{localinc}, for any $G \in S_i(M)$ and each $F \in \phi^{-1}(G) \subseteq S_i(K)$,
there is a bijection also denoted by $\phi$:
$$\phi: \{\bar{F} \in S_{i+1}(K): F \in \p \bar{F}\} \to \{\bar{G} \in S_{i+1}(M): G \in \p \bar{G}\}, \bar{F} \mapsto \phi(\bar{F}).$$
Similarly,
there is a bijection:
 $$\phi: \{F' \in S_i(K): F \cup F' \in S_{i+1}(K)\} \to \{G' \in S_{i}(M): G \cup G' \in S_{i+1}(M)\}, F' \mapsto \phi(F').$$
Using these bijections, the weight assumption and the notations $G=\phi(F)$, $G':=\phi(F')$ and $\bar{G}:=\phi(\bar{F})$, we have
\begin{align*}
(\L_i^{up}(K) \bar{f})([F]) &=\sum_{\bar{F} \in S_{i+1}(K): \atop F \in \p \bar{F}} \frac{w(\bar{F})}{w(F)}\bar{f}([F])+ \sum_{F' \in S_i(K): \atop F' \cup F= \bar{F} \in S_{i+1}(K)} \frac{w(\bar{F})}{w(F)}\sgn([F], \p [\bar{F}]) \sgn([F'], \p [\bar{F}])\bar{f}([F'])\\
&=\sum_{\bar{F} \in S_{i+1}(K): \atop F \in \p \bar{F}} \frac{w(\phi(\bar{F}))}{w(\phi(F))}f([\phi(F)])\sgn([F], [\phi(F)])\\
&~~~+ \sum_{F' \in S_i(K): \atop F' \cup F= \bar{F} \in S_{i+1}(K)} \frac{w(\phi(\bar{F}))}{w(\phi(F))}\sgn([F], \p [\bar{F}]) \sgn([F'], \p [\bar{F}])f([\phi(F')])\sgn([F'], [\phi(F')])\\
&=\sum_{\bar{G} \in S_{i+1}(M): \atop G \in \p \bar{G}} \frac{w(\bar{G})}{w(G)}f([G])\sgn([F], [G])\\
&~~~+ \sum_{G' \in S_i(M): \atop G \cup G' = \bar{G} \in S_{i+1}(M)} \frac{w(\bar{G})}{w(G)}\sgn([F], \p [\bar{F}]) \sgn([F'], \p [\bar{F}])f([G'])\sgn([F'], [G'])\\
&=\sgn([F], [G]) \\
& ~~~ \cdot \bigg(\!\sum_{\bar{G} \in S_{i+1}(M): \atop G \in \p \bar{G}} \frac{w(\bar{G})}{w(G)}f([G])
+ \!\!\!\!\!\!\!\!\sum_{G' \in S_i(M): \atop G \cup G' = \bar{G} \in S_{i+1}(M)} \!\!\!\! \frac{w(\bar{G})}{w(G)}\sgn([G], \p [\bar{G}]) \sgn([G'], \p [\bar{G}])f([G'])\!\bigg)\\
&=\sgn([F], [G])(\L_i^{up}(M) f)([G])=0,
\end{align*}
where the last equality follows from (\ref{kerup}) and the second last equality follows from the sign relation (\ref{sgncov}).
Thus, $\bar{f} \in \text{ker}\L_i^{up}(K)$.

(2) If $f \in \text{ker}\L^{\dw}_i(M)$, then for each $G \in S_i(M)$,
$(\L^{\dw}_i(M)f([G])=0$.
For every $G \in S_i(M)$ and each $F \in \phi^{-1}(G) \subseteq S_i(K)$,
As $\phi|_F: F \to G$ is bijection, there is a bijection
$$\phi: \{E \in S_{i-1}(K): E \in \p F\} \to \{H \in S_{i-1}(M): H \in \p G\}, E \mapsto \phi(E).$$
By Lemma \ref{localinc},
there is a bijection
  $$\phi: \{F' \in S_i(K): F \cap F' \in S_{i-1}(K)\} \to \{G'\in S_i(M): G \cap G'\in S_{i-1}(M)\}, F' \mapsto \phi(F').$$
Using these bijections,  the weight assumption and the notations $G=\phi(F)$, $G':=\phi(F')$ and $H:=\phi(E)$, we have
\begin{align*}
(\L_i^{\dw}(K) \bar{f})([F]) &=\sum_{E \in S_{i-1}(K): \atop E \in \p F} \frac{w(F)}{w(E)}\bar{f}([F])+ \sum_{F' \in S_i(K): \atop F \cap F'=E \in S_{i-1}(K)} \frac{w(F')}{w(E)}\sgn([E], \p [F) \sgn([E], \p [F'])\bar{f}([F'])\\
&=\sum_{E \in S_{i-1}(K): \atop E \in \p F} \frac{w(\phi(F))}{w(\phi(E))}f([\phi(F)])\sgn([F],[\phi(F)])\\
&+ \sum_{F' \in S_i(K): \atop F \cap F'=E \in S_{i-1}(K)} \frac{w(\phi(F'))}{w(\phi(E))}\sgn([E], \p [F]) \sgn([E], \p [F'])f([\phi(F')])\sgn([F'],[\phi(F')])\\
&=\sum_{H \in S_{i-1}(M): \atop H \in \p G} \frac{w(G)}{w(H)}f([G])\sgn([F], [G])\\
&~~~+ \sum_{G' \in S_i(M):\atop G \cap G' =H \in S_{i-1}(M)} \frac{w(G')}{w(E)}\sgn([E], \p [F]) \sgn([E], \p [F'])f([G'])\sgn([F'],[G'])\\
&=\sgn([F], [G]) \\
& ~~~ \cdot  \bigg(\!\!\sum_{H \in S_{i-1}(M): \atop H \in \p G} \!\!\frac{w(G)}{w(H)}f([G])
+ \!\!\!\!\!\!\!\!\!\!\!\!\!\!\sum_{G' \in S_i(M):\atop G \cap G' =H \in S_{i-1}(M)} \!\!\!\!\!\!\!\!\!\!\!\!\frac{w(G')}{w(E)}\sgn([H], \p [G) \sgn([H], \p [G'])f([G'])\!\!\bigg)\\
&=\sgn([F], [G])(\L^{\dw}_i(M)f([G])=0,
\end{align*}
where the second last equality follows from  (\ref{sgncov}).
Thus, $\bar{f} \in \text{ker}\L_i^{\dw}(K)$.
 \end{proof}

\begin{thm}\label{dimhom}
Let $K$ be a covering  complex of a  complex $M$.
Then $\beta_i(K) \ge \beta_i(M)$ for all $i$.
\end{thm}

\begin{proof}
By Eckmann's result, it suffices to prove the $\ker \L_i(M)$ can be embedded into $\text{ker}\L_i(K)$ as a subspace.
By Theorem 2.2 of \cite{HJ13B}, $\ker \L_i(M)= \ker \L_i^{\up}(M) \cap \ker \L_i^{\dw}(M)$.
Suppose that $\phi: K \to M$ is a covering map.
We choose weights on $K,M$ such that
$\frac{w_K(\bar{F})}{w_K(F)}=\frac{w_M(\phi(\bar{F}))}{w_M(\phi(F))}$ for each pair $(F,\bar{F})$ with $F \in \p \bar{F}$.
This can easily be done by using combinatorial Laplace operator or normalized Laplace operator.
By Lemma \ref{kerupdown}, for each $f \in \ker \L_i^{\up}(M) \cap \ker \L_i^{\dw}(M)=\ker \L_i(M)$, letting $\bar{f}$ be defined as in (\ref{barf}),
then $\bar{f} \in \ker \L_i^{\up}(K) \cap \ker \L_i^{\dw}(K)=\ker \L_i(K)$.

Let $V=\{\bar{f}: f \in \ker \L_i(M)\}$.
It is easily verified that $V$ is a subspace of $\ker \L_i(K)$.
We will show that $V$ has the same dimension as $\ker \L_i(M)$.
Let $f_1,\ldots,f_t$ be a basis of $\ker \L_i(M)$.
Let $\la_1,\ldots,\la_t \in \R$ be such that
$$ \la_1 \bar{f}_1+\cdots+\la_t \bar{f}_t=0.$$
For each $G \in S_i(M)$, choosing an $F \in \phi^{-1}(G)$, we have
\begin{equation}
\begin{split}
\la_1 (\bar{f}_1)([F])+\cdots+\la_t (\bar{f}_t)([F]) & =
\la_1 f_1([G]) \sgn([F],[G])+\cdots+\la_t f_t([G]) \sgn([F],[G])\\
& = \sgn([F],[G]) \left(\la_1 f_1([G])+\cdots+ \la_t f_t([G])\right)\\
& =0.
\end{split}
\end{equation}
So we have
$$ \la_1 f_1+ \cdots+\la_t f_t =0$$
which implies that $\la_1=\cdots=\la_t=0$ as $f_1,\ldots,f_t$ are linearly independent.
So $\bar{f}_1,\ldots,\bar{f}_t$ is a basis of $V$.
The result follows.
\end{proof}

\begin{exm}
Let $K,M$ be the complexes in Fig. \ref{KL}.
It is easy to see
$$ \beta_0(K)=\beta_0(M)=0, \beta_1(K)=\beta_1(M)=1, \beta_2(K)=\beta_2(M)=0.$$

Let $G,\bar{G}$ be the graphs in Fig. \ref{graphcov}.
It is seen that $\bar{G}$ is a $2$-fold covering of $G$, and
$$ \beta_0(G)=\beta_0(\bar{G})=0, \beta_1(G) = 2 < \beta(\bar{G})=3.$$
So it will be interesting to characterize the equality case of Theorem \ref{dimhom}.

\begin{figure}[h]
\centering
\includegraphics[scale=.75]{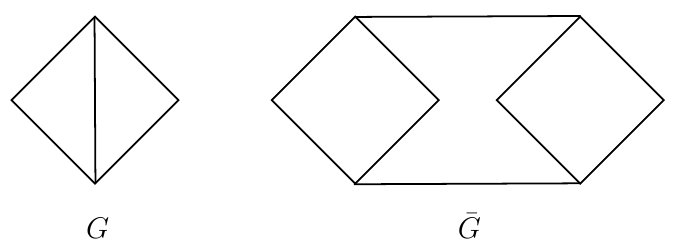}~~~
\caption{\small A graph $G$ and its $2$-fold covering $\bar{G}$}\label{graphcov}
\end{figure}

\end{exm}

\end{document}